\let\oldvec\vec
\documentclass[smallextended]{svjour3}
\let\vec\oldvec     
\usepackage{amsmath,amsfonts,amsopn,color,amssymb,graphics,graphicx}

\usepackage{epsfig}

\usepackage{pstricks}
\usepackage{pstricks-add}

\usepackage{cite}
\usepackage{amscd}
\usepackage{latexsym}
\usepackage{mathrsfs}
\usepackage{pst-all}

\usepackage{srcltx}
\usepackage{psfrag}
\usepackage{epstopdf}
\usepackage{multirow}
\usepackage{subcaption}
\usepackage{lmodern}

\smartqed

\journalname{}
\begin{document}

\title{Two-parameter regularization of ill-posed spherical pseudo-differential equations in the space of continuous functions
}

\titlerunning{Two-parameter regularization}        

\author{Hui Cao        \and
        Sergei V. Pereverzyev \and
				Ian~H.~Sloan \and
				Pavlo Tkachenko
}


\institute{H. Cao \at
							Guangdong Province Key Laboratory of Computational Science, Sun Yat-sen University, Guangzhou, China \\
              \email{caohui6@mail.sysu.edu.cn}           
           \and
         S. V. Pereverzyev \at
              Johann Radon Institute for Computational and Applied Mathematics, Austrian Academy of Sciences, Altenbergerstrasse 69, 4040 Linz, Austria \\
              \email{	sergei.pereverzyev@oeaw.ac.at} 							
							\and
         I. H. Sloan \at
              School of Mathematics and Statistics, University of New South Wales, Sydney NSW 2052, Australia \\
              \email{I.Sloan@unsw.edu.au} 						
						 \and
         P. Tkachenko \at
              Johann Radon Institute for Computational and Applied Mathematics, Austrian Academy of Sciences, Altenbergerstrasse 69, 4040 Linz, Austria \\
              \email{pavlo.tkachenko@oeaw.ac.at}
}

\date{Received: date / Accepted: date}

\maketitle

\begin{abstract}

In this paper, a two-step regularization method is used to solve an
ill-posed spherical pseudo-differential equation in the presence of noisy
data. For the first step of regularization we approximate the data by
means of a spherical polynomial that minimizes a functional
with a penalty term consisting of the squared norm in a Sobolev
space. The second step is a regularized collocation method. An error
bound is obtained in the uniform norm, which is potentially smaller than that for
either the noise reduction alone or the regularized collocation alone. We
discuss an \textit{a posteriori} parameter choice, and present some
numerical experiments, which support the claimed superiority of the
two-step method.

\end{abstract}

\section{Introduction}
Mathematical models appearing in geoscience commonly have the form of an ill-posed spherical pseudo-differential equation

\begin{equation} \label{eq:1}
	Ax=y,
\end{equation}
where $A$ is a pseudo-differential operator that relates continuous
functions $x\in C(\Omega_R)$  and $y\in C(\Omega_\rho)$ defined on
concentric spheres $\Omega_R, \Omega_\rho \in \mathbb{R}^3$ of radii
$R\leq\rho$. For example, in satellite geodesy, this approach has been
introduced in \cite{F1999, S1983} where the spheres $\Omega_R,
\Omega_\rho$ are models of respectively the surface of the Earth
and the surface traversed by the satellite.

Recall that a spherical pseudo-differential operator
$A:C(\Omega_R)\rightarrow C(\Omega_\rho)$ is a linear operator that
assigns to any $x\in C(\Omega_R)$ a function

\begin{equation} \label{eq:2}
	 Ax:=\sum^{\infty}_{k=0}a_k\sum^{2k+1}_{j=1}\widehat{x}_{k,j}\frac{1}{\rho}Y_{k,j}\left(\frac{\cdot}{\rho}\right)\ \in C(\Omega_\rho),
\end{equation}
where

	\[
	\widehat{x}_{k,j}=\left\langle \frac{1}{R}Y_{k,j}\left(\frac{\cdot}{R}\right), x(\cdot)\right\rangle_{L_2(\Omega_R)}:=\frac{1}{R}\int_{\Omega_R}x(\tau)Y_{k,j}\left(\frac{\tau}{R}\right)d\Omega_R(\tau)
	\] are the spherical Fourier coefficients, and $Y_{k,j}(\cdot),
j=1,2,...,2k+1$, are the spherical harmonics \cite{M1966} of degree $k$
which are $L_2$-orthonormal with respect to the unit sphere $\Omega_1\in
\mathbb{R}^3$, as a result of which

\begin{equation}\label{eq:orthog}
\left\langle
\frac{1}{R}Y_{k,j}\left(\frac{\cdot}{R}\right),\frac{1}{R}Y_{k',j'}\left(\frac{\cdot}{R}\right)\right\rangle_{L_2(\Omega_R)}=\left\langle
Y_{k,j},Y_{k',j'}\right\rangle_{L_2(\Omega_1)}=\delta_{kk'}\delta_{jj'}.
\end{equation}
The sequence of real numbers $\left(a_k\right)_{k=0}^\infty$ is
referred to as the spherical symbol of $A$.  We shall assume that
$a_k$ is positive, and converges monotonically to zero.

In the case when the symbol sequence $a_k$ tends to zero fast enough
the operator $A$ is compact. Therefore, its inverse $A^{-1}$ is unbounded,
and keeping in mind Hadamard's definition of a well-posed problem
(existence, uniqueness, and continuity of inverse), we conclude that for
this case the first and the third conditions are violated, and the problem
(\ref{eq:1}), (\ref{eq:2}) with $y\in C(\Omega_R)$ becomes ill-posed.
Therefore, a regularization technique should be employed for solving it
\cite{E1996}.

As examples of the ill-posed problem (\ref{eq:1}), (\ref{eq:2}) we can
mention the satellite-to-satellite tracking problem (SST-problem) with
$a_k=\frac{k+1}{\rho}\left(\frac{R}{\rho}\right)^k$, the satellite gravity
gradiometry problem (SGG-problem) with
$a_k=\frac{(k+1)(k+2)}{\rho^2}\left(\frac{R}{\rho}\right)^k$, etc. (for
more details on these and other examples we can refer an interested reader
to \cite{F1999}).  These problems are severely ill-posed because of
the occurrence of the geometric factor $(R/\rho)^k$.

It is worth mentioning that many practical applications use
a finite dimensional approximation of the solution of (\ref{eq:1}),
(\ref{eq:2}). For example, Earth Gravity Models, such as EGM96 or EGM2008 \cite{P2008} are
parametrized by the spherical Fourier coefficients up to some prescribed degree
$M$.

Note that in applications the function $y$ is assumed to be continuous.
However, in practice one is provided just with a finite number of points
$\{t_{i}\}_{i=1}^{N}\subset \Omega_\rho$ at which information about the
values of $y$ is collected. It should be noted also that the pointwise
data $y^\epsilon(t_i)$ contain measurement errors, which can be
modeled, for example, in the following way:

\[
    \left|y^\epsilon(t_i)-y(t_i)\right|\leq \epsilon_{i},\ i=1,\ldots,N.
\]
We assume for convenience that there exists a function $y^\epsilon\in
C(\Omega_\rho)$ standing for the noisy version of the original function
$y$, such that

$$\|y^\epsilon-y\|_{C(\Omega_\rho)}\leq \epsilon:=\max_{1\leq i\leq N} \{\epsilon_{i}\},$$
where $\epsilon_i$ are measurement errors.

In such a setup the problem (\ref{eq:1}), (\ref{eq:2}) is reduced to the
following spherical pseudo-differential operator equation
\begin{equation} \label{eq:3}
A_Mx:=\sum^M_{k=0}a_k\sum^{2k+1}_{j=1}\widehat{x}_{k,j}\frac{1}{\rho}Y_{k,j}\left(\frac{\cdot}{\rho}\right)=y^\epsilon.
\end{equation}

Several regularization techniques can be used for treating (\ref{eq:3}).
For example, the method where the discretization level $M$ plays the role
of regularization parameter was discussed in \cite{B1996, B2003, H2003}.
However, in our case we assume that the value of $M$ is prescribed,
thus this approach cannot be used.

The problem (\ref{eq:3}), but without noise in the right-hand side, was studied in \cite{L2006}. In this noise-free case the solution $x$ can be approximated by solving the equation

$$A_Mx=\widehat{V}_My,$$
where $\widehat{V}_M:C(\Omega_\rho)\rightarrow\mathbb{P}_M(\Omega_\rho)$
is the so-called quasi-interpolatory operator. Here by
$\mathbb{P}_M(\Omega_\rho)$  we denote the set of all spherical
polynomials of degree less than or equal to $M$, or in other words the
restriction to $\Omega_\rho$ of the polynomials in $\mathbb{R}^3$ of
degree less than or equal to $M$. Thus, in \cite{L2006} the authors
suggest constructing a polynomial approximation of $y$ from the original
pointwise data $y(t_i), i=1,2,\ldots,N$, and then formally inverting
the operator $A_M$.

In principle, this idea can be used also for the ill-posed case
(\ref{eq:3}). However, in contrast to the approximation of a noise-free
continuous function $y$ on the sphere by means of polynomials, which has
been discussed by many authors (see, for example, \cite{G1997, R2003,
S2011, S1979, W1981}), the approximation of noisy functions $y^\epsilon$
has been studied only recently in \cite{A2012, P2014}, to the best of our
knowledge. Applying the method from \cite{A2012, P2014} we first
perform so-called data-smoothing (or noise reduction). After this
data preprocessing the formal inversion of $A_M$ should be safer.
However, for performing the noise reduction step one needs \textit{a
priori} information about the smoothness of the function $y$, which is
usually not available, a fact that makes the direct application of the
scheme \cite{P2014} not always appropriate.

In a different direction, for estimating the Fourier coefficients
$\widehat{x}_{k,j}$ directly from noisy measurements $y^\epsilon(t_i)$ a
regularized collocation method has been recently presented in
\cite{N2013}. This method is based on the standard and widely used
Tikhonov-Phillips regularization. However, it is well-known that the
Tikhonov-Phillips method suffers from saturation \cite{M2004, N1997},
meaning that the accuracy of reconstruction cannot be improved regardless
of the smoothness of the solution $x$. Another point is that while
the Tikhonov-Phillips method has been well studied in the
Hilbert space $L_2$, to the best of our knowledge, no analysis
has been done in the space of continuous functions, a natural choice
for our noise model.

In the present study we combine these two approaches.
Moreover, in contrast to the previous results, we will analyze the
approximation in the space of continuous functions. A combination of two
regularization methods into one can be seen as a two-step regularization,
in which we use the composition $R_{\alpha,M}\circ
T_{\lambda,M}$ of data smoothing operator $T_{\lambda,M}$, and a
regularized collocation operator $R_{\alpha,M}$. In the literature there
are not so many studies on two-step regularization, and we can
refer only to \cite{K2006, K2008}. However, the analysis in those
papers does not correspond to the setting of our problem
(\ref{eq:1})--(\ref{eq:3}).

The paper is organized as follows.  In the next section we present
the regularization method for noise reduction and define the data
smoothing operator $T_{\lambda,M}$. In Section 3 we will give a short
overview of the regularized collocation method, and define the operator
$R_{\alpha,M}$. Section 4 is devoted to theoretical error bounds for
the constructed two-parameter regularization. We will show that
the approximation has the potential to perform at least as well
as the better of the one-parameter regularizations which are involved in
the composition. Finally, in the last section we discuss an \textit{a
posteriori} parameter choice rule, and present some numerical
experiments supporting the claimed superiority of our method.

\section{Data noise reduction}
At the first step of our scheme we approximate the noisy continuous
function $y^\epsilon\in C(\Omega_\rho)$ by means of a spherical polynomial
$p_M \in \mathbb{P}_M(\Omega_\rho)$. As discussed in the
Introduction, instead of $y$ we are provided only with pointwise
measurements $y^\epsilon(t_i), i=1,2,\ldots,N$. Therefore we introduce the
sampling operator $S_N:C(\Omega_\rho)\rightarrow \mathbb{R}^N_\omega$ for
which

$$S_Ny^\epsilon:=(y^\epsilon(t_1), y^\epsilon(t_2),\ldots,y^\epsilon(t_N)).$$
By $\mathbb{R}^N_\omega$ we denote the vector space $\mathbb{R}^N$
equipped with the inner product

$$\left\langle \eta,\gamma\right\rangle_\omega:=\sum_{i=1}^N\omega_i\eta_i\gamma_i,\ \eta, \gamma \in \mathbb{R}^N;$$
the corresponding norm is
$\left\|\eta\right\|_\omega=\left\langle
\eta,\eta\right\rangle_\omega^{1/2}$. Here
$\omega_1,\omega_2,\ldots,\omega_N$ are positive weights in a cubature
formula, and $t_1, t_2,\ldots,t_N\in \Omega_\rho$ are the corresponding
cubature points, with the cubature rule having the property of being
exact for all polynomials of degree up to $2M$,
\begin{equation}\label{eq:5}
\forall p\in\mathbb{P}_{2M}(\Omega_\rho),\ \ \sum_{i=1}^N \omega_ip(t_i)=\int_{\Omega_\rho}p(\zeta)d\Omega_\rho(\zeta).
\end{equation}
A method for approximating $y$ from $S_Ny^\epsilon\in\mathbb{R}^N_\omega$
based on such a cubature rule (but with all weights equal) was
proposed recently in \cite{A2012}, while its generalization to
other positive weights was presented in \cite{P2014}.

We briefly outline the method from \cite{P2014} and its extension
to pseudo-differential operator equations. We start with the observation
that the space $\mathbb{P}_M(\Omega_\rho)$ of spherical polynomials
$p$ of degree at most $M$ can be considered as the reproducing kernel
Hilbert space (RKHS) $\mathcal{H}_K$ generated by the kernel
\begin{equation}\label{eq:6}
K(t,\tau)=\sum_{k=0}^M \beta_k^{-2}\sum_{j=1}^{2k+1}\frac{1}{\rho} Y_{k,j}\left(\frac{t}{\rho}\right) \frac{1}{\rho} Y_{k,j}\left(\frac{\tau}{\rho}\right),\ t,\tau\in\Omega_\rho,
\end{equation}
where $\beta=(\beta_1,\beta_2,\ldots,\beta_M,\ldots)$ is a non-decreasing
sequence of positive parameters. Note that the inner product of
$\mathcal{H_K}$ associated with this kernel can be written as
\begin{equation}\label{eq:inner}
\left\langle f,g \right\rangle_{\mathcal{H}_K}=\sum_{k=0}^M \frac{\beta_k^2}{\rho^2}\sum_{j=1}^{2k+1} \left\langle Y_{k,j}\left(\frac{\cdot}{\rho}\right),f\right\rangle_{L^2(\Omega_\rho)} \left\langle Y_{k,j}\left(\frac{\cdot}{\rho}\right),g\right\rangle_{L^2(\Omega_\rho)}.
\end{equation}
The reproducing property $\langle
f,K(\cdot,\tau)\rangle_{\mathcal{H}_K}=f(\tau)$ for $\tau \in \Omega_R$ can be verified easily. In
this paper the numbers $\beta_1,\beta_2,\ldots,\beta_N$ are assumed to be
given \textit{a priori}. We shall see their role later.

The approximation $p_{\min}$ studied in \cite{A2012, P2014} appears
as the minimizer of the following functional

\begin{equation}\label{eq:7}
p_{\min}=\arg \min\left\{\left\|S_Np-S_Ny^\epsilon\right\|_\omega^2+\lambda\left\|p\right\|^2_{\mathcal{H}_K},\ p\in \mathbb{P}_M(\Omega_\rho)\right\},
\end{equation}
where $\lambda\geq 0$ is a regularization parameter. Note that
(\ref{eq:7}) can be seen as the definition of the data smoothing operator
$T_{\lambda,M}:C(\Omega_\rho)\rightarrow \mathbb{P}_M(\Omega_\rho)$ such
that $p_{\min}=T_{\lambda,M}y^\epsilon$.  Note that the regularization
term involves the $\mathcal{H}_K$-norm of $p$, and hence depends on the
rate of growth of the parameters $\beta_k$. The solution of (\ref{eq:7})
is given by the following theorem.

\begin{theorem} Assume that the points $\left\{t_i\right\}$ and weights
 $\left\{\omega_i\right\}$ are such that (\ref{eq:5}) holds.
 Then the minimizer $T_{\lambda,M}y^\epsilon=p_{\min}$ in (\ref{eq:7}) has
the form

\begin{equation}\label{eq:8}
\left(T_{\lambda,M}y^\epsilon\right)(\cdot)= \sum_{k=0}^M \frac{1}{1+\lambda\beta_k^2}\sum_{j=1}^{2k+1} \frac{1}{\rho}Y_{k,j}\left(\frac{\cdot}{\rho}\right)\sum_{i=1}^N \omega_i\frac{1}{\rho}Y_{k,j}\left(\frac{t_i}{\rho}\right)y^\epsilon(t_i).
\end{equation}
\end{theorem}

\begin{proof}

It is known that the minimizer of the  functional in (\ref{eq:7}) can
be written as

\begin{equation}\label{eq:9}
T_{\lambda,M}y^\epsilon=(\lambda I+S_N^*S_N)^{-1}S_N^*S_Ny^\epsilon,
\end{equation}
where $S_N^*:\mathbb{R}^N_\omega \rightarrow \mathcal{H}_K$ is the adjoint
of $S_N$, and $I$ is the identity operator in $\mathcal{H}_K$. By
definition, for $(\eta_i)_{i=1}^N\in\mathbb{R}^N_\omega$ we have
\begin{eqnarray*}
\left\langle \eta, S_Nf\right\rangle_\omega&=&\sum_{i=1}^N\eta_i f(t_i)\omega_i=\sum_{i=1}^N\eta_i \left\langle K(t_i,\cdot),f(\cdot)\right\rangle_{\mathcal{H}_K} \omega_i \\
&=&\left\langle \sum_{i=1}^N\omega_iK(t_i,\cdot)\eta_i,f(\cdot)\right\rangle_{\mathcal{H}_K},
\end{eqnarray*}
thus there holds

$$(S_N^*\eta)(\cdot)=\sum_{i=1}^N\omega_iK(t_i,\cdot)\eta_i, \ \forall \eta \in \mathbb{R}^N_\omega,$$
and thereby

$$(S_N^*S_N f)(\cdot)=\sum_{i=1}^N\omega_iK(t_i,\cdot)f(t_i).$$

Inserting this expression into (\ref{eq:9}), we observe that $p_{\min}$
solves

\begin{equation}\label{eq:10}
\lambda p_{\min}(\cdot)+\sum_{i=1}^N\omega_iK(t_i,\cdot)p_{\min}(t_i)=\sum_{i=1}^N\omega_iK(t_i,\cdot)y^\epsilon(t_i)
\end{equation}
The spherical harmonic expansion of the polynomial
$p_{\min}=T_{\lambda,M} y^{\epsilon}$ is
\begin{equation}\label{eq:pmin}
p_{\min}(t)=\sum_{k=0}^M\sum_{j=1}^{2k+1} \left\langle p_{\min},\frac{1}{\rho}Y_{k,j}\left(\frac{\cdot}{\rho}\right)\right\rangle_{L^2(\Omega_\rho)}\frac{1}{\rho}Y_{k,j}\left(\frac{t}{\rho}\right),\quad t\in\Omega_{\rho}.
\end{equation}
To find the coefficients in this expansion, we write the second term
on the left-hand side of \eqref{eq:10}, on using \eqref{eq:5}
and then \eqref{eq:6}, as
\begin{eqnarray*}
\sum_{i=1}^N\omega_i K(t_i,t)p_{\min}(t_i)&=&\left\langle p_{\min}, K(\cdot,t)\right\rangle_{L^2(\Omega_\rho)}\\
&=&\sum_{k=0}^N\beta_k^{-2}\sum_{j=1}^{2k+1}\left\langle p_{\min},\frac{1}{\rho}Y_{k,j}\left(\frac{\cdot}{\rho}\right)\right\rangle_{L_2(\Omega_\rho)}\frac{1}{\rho}Y_{k,j}\left(\frac{t}{\rho}\right).
\end{eqnarray*}
After expanding the right-hand side of \eqref{eq:10} using \eqref{eq:6},
and then equating the coefficients of
$Y_{k,j}\left(\frac{t}{\rho}\right)$, we obtain
\[
(\lambda+\beta_k^{-2})\left\langle p_{\min},Y_{k,j}\left(\frac{t}{\rho}\right)\right\rangle_{L^2(\Omega_\rho)}=\beta_k^{-2}\sum_{i=1}^N\omega_i Y_{k,j}\left(\frac{t}{\rho}\right)y^\epsilon(t_i),\; t\in\Omega_\rho,
\]
which on substituting into \eqref{eq:pmin} yields the desired result. 

\qed
\end{proof}

At this point the solution of the first step depends on the regularization
parameter $\lambda$, and the penalization weights $\beta_k$. As
mentioned in the Introduction, the choice of the regularization
parameter $\lambda$ will be addressed in the last section. A
data-driven choice of the penalization weights $\beta_k$ has been recently
discussed in \cite{P2014}.

The assumption that the function $x$ on $\Omega_R$ is continuous implies
that $x \in L^2(\Omega_R)$, and hence that its Fourier coefficients
$\left\langle
\frac{1}{R}Y_{k,j}\left(\frac{\cdot}{R}\right),x\right\rangle_{L_2(\Omega_R)}$
with respect to the basis of spherical harmonics are square-summable, i.e,

\[
\sum_{k=0}^\infty\sum_{j=1}^{2k+1}\left|\left\langle \frac{1}{R}Y_{k,j}\left(\frac{\cdot}{R}\right),x\right\rangle_{L_2(\Omega_R)}\right|^2<\infty.
\]
Any additional smoothness of $x$ can be measured in terms of the
summability of Fourier coefficients with some increasing weights depending
on the sequence $\left(\beta_k\right)$ or, as it is usual for the
regularization theory (see, e. g., \cite{L2013}), on the symbol
$\left(a_k\right)$. Therefore, it is convenient to introduce two
Sobolev spaces $W^{\phi,\beta}$, and $W^{\psi,a}$, the first
depending on the sequence $(\beta_k)$ and the second on the symbol
$(a_k)$, and  defined by

\begin{small}
\begin{eqnarray}
W^{\phi,\beta}:=\left\{g\in L^2(\Omega_R):\sum_{k=0}^\infty\sum_{j=1}^{2k+1}\frac{\left|\left\langle \frac{1}{R}Y_{k,j}\left(\frac{\cdot}{R}\right),g\right\rangle_{L_2(\Omega_R)}\right|^2}{\phi^2(\beta_k^{-2})}=:\|g\|_{W^{\phi,\beta}}^2<\infty\right\}, \label{eq:14}  \\
W^{\psi,a}:=\left\{g\in L^2(\Omega_R):\sum_{k=0}^\infty\sum_{j=1}^{2k+1}\frac{\left|\left\langle \frac{1}{R}Y_{k,j}\left(\frac{\cdot}{R}\right),g\right\rangle_{L_2(\Omega_R)}\right|^2}{\psi^2(a_k^2)}=:\|g\|_{W^{\psi,a}}^2<\infty\right\} \label{eq:15},
\end{eqnarray}
\end{small}
where $\phi, \psi$ are non-decreasing functions such that $\phi(0)=0$
and $\psi(0)=0$. In the literature, see, e.g., \cite{L2013}, the
functions $\phi, \psi$ go under the name of ``smoothness index
functions'', and are usually unknown.

\section{Regularized collocation method after noise reduction}

After the first step of our method the reduced original equation
(\ref{eq:3}) is transformed into the equation

\begin{equation}\label{eq:16}
A_M x=T_{\lambda,M}y^\epsilon.
\end{equation}
The regularized collocation method (see \cite{N2013}) is now
applied to this equation, yielding an approximate solution
$x_{\alpha,\lambda}^\epsilon$ of the equation (\ref{eq:16}),
defined as the minimizer of the functional

\begin{equation}\label{eq:17}
\left\|B_{N,M} x-S_NT_{\lambda,M}y^\epsilon\right\|_\omega^2+\alpha\left\|x\right\|^2_{L_2(\Omega_R)},
\end{equation}
where $B_{N,M}:=S_NA_M:L_2(\Omega_R)\rightarrow \mathbb{R}_\omega^N$, and
$\alpha\geq 0$ is the regularization parameter for the second step.
The minimizer of (\ref{eq:17}) can be written in the form

\begin{equation}\label{eq:18}
x_{\alpha,\lambda}^\epsilon=(\alpha I+B_{N,M}^*B_{N,M})^{-1}B_{N,M}^*S_NT_{\lambda,M}y^\epsilon,
\end{equation}
where $B_{N,M}^*:\mathbb{R}_\omega^N\rightarrow L_2(\Omega_R)$ is the
adjoint of $B_{N,M}$, given by

$$(B_{N,M}^*\eta)(\cdot)=\sum_{k=0}^M a_k\sum_{j=1}^{2k+1}\frac{1}{R}Y_{k,j}\left(\frac{\cdot}{R}\right)\sum_{i=1}^N \omega_i \frac{1}{\rho} Y_{k,j}\left(\frac{t_i}{\rho}\right)\eta_i,$$
from which it follows, using \eqref{eq:3}, that

\[
B_{N,M}^*B_{N,M} x=\sum_{k=0}^M a_k^2\sum_{j=1}^{2k+1}\frac{1}{R}Y_{k,j}\left(\frac{\cdot}{R}\right)\widehat{x}_{k,j}.
\]
Using (\ref{eq:18}) and (\ref{eq:8}) we then obtain explicitly

\begin{equation}\label{eq:19}
x_{\alpha,\lambda}^\epsilon:=\sum_{k=0}^M \frac{a_k}{\alpha+a_k^2}\frac{1}{1+\lambda\beta_k^2}  \sum_{j=1}^{2k+1}\frac{1}{R}Y_{k,j}\left(\frac{\cdot}{R}\right)\sum_{i=1}^N \omega_i\frac{1}{\rho}Y_{k,j}\left(\frac{t_i}{\rho}\right)y^\epsilon(t_i).
\end{equation}

If we now define
$R_{\alpha,M}:\mathbb{P}_{M}(\Omega_\rho)\rightarrow
\mathbb{P}_M(\Omega_R)$ as
\begin{eqnarray}
\left(R_{\alpha,M} p\right)(\cdot)&:=& \left( \left(\alpha I+B_{N,M}^{\ast}B_{N,M} \right)^{-1}B_{N,M}^{\ast}S_{N} p\right)(\cdot) \nonumber\\
&=&\sum_{k=0}^{M}\frac{a_{k}}{\alpha+a_{k}^{2}}\sum_{j=1}^{2k+1}\frac{1}{R}Y_{k,j}\left(\frac{\cdot}{R}\right)
\sum_{i=1}^{N}\omega_{i}\frac{1}{\rho}
Y_{k,j}\left(\frac{t_{i}}{\rho}\right) p(t_{i}) \label{eq:Ra}\\
&=&\sum_{k=0}^{M}\frac{a_{k}}{\alpha+a_{k}^{2}}\sum_{j=1}^{2k+1}\frac{1}{R}Y_{k,j}\left(\frac{\cdot}{R}\right)
\left\langle\frac{1}{\rho}Y_{k,j},p\right\rangle_{L_2(\Omega_\rho)}.
\end{eqnarray}
then we can write
$x_{\alpha,\lambda}^\epsilon=R_{\alpha,M}T_{\lambda,M}y^\epsilon$.
Thus $R_{\alpha,M}$ reflects the regularized collocation second step
of the method, whereas $T_{\lambda,M}$ corresponds to the noise reduction
first step.

As one can see from (\ref{eq:19}), in the end we approximate the
unknown solution $x$ on $\Omega_R$ by means of a spherical
polynomial of degree $M$. In this context a question arises about the
best approximation of a continuous function $x$ by means of a
polynomial of degree $M$. In turn the quality of best polynomial
approximation is determined by the smoothness of $x$.  We shall assume
that $x$ lies in the intersection of $W^{\phi,\beta}$ and $W^{\psi,a}$,
see (\ref{eq:14}) or (\ref{eq:15}), thus the smoothness of $x$ is encoded
in $\phi$ and $\beta_k$, or $\psi$ and $a_k$. For example, if the
smoothness index function $\phi(t)$ and the sequence
$\beta=\left\{\beta_k\right\}$ increase polynomially with $t$ and $k$,
then Jackson's theorem on the sphere (see \cite{R1971}, Theorem 3.3) tells
us that for $x\in W^{\phi,\beta}$ there is $\nu>0$ such that

\begin{equation}\label{eq:20}
\inf_{p\in\mathbb{P}_M}\left\|x-p\right\|_{C(\Omega_R)}=O\left(M^{-\nu}\right).
\end{equation}

At the same time, if the sequence $\beta=\left\{\beta_k\right\}$ increases exponentially then for polynomially increasing $\phi$ and $x\in W^{\phi,\beta}$ we have

\[
\inf_{p\in\mathbb{P}_M}\left\|x-p\right\|_{C(\Omega_R)}=O\left(e^{-qM}\right),
	\]
where $q$ is some positive number that does not depend on $M$.

In the error analysis of the next section we shall make use of a
constructive polynomial approximation introduced in \cite{S2011}, in which
$x\in C(\Omega_R)$ is approximated by $V_M x\in\mathbb{P}_M(\Omega_R)$,
given by

\begin{equation}\label{VMx}
\left(V_Mx\right)(t)=\sum_{k=0}^M h\left(\frac{k}{M}\right)\sum_{j=1}^{2k+1}\frac{1}{R}Y_{k,j}\left(\frac{t}{R}\right)\left\langle \frac{1}{R} Y_{k,j}\left(\frac{\cdot}{R}\right),x\right\rangle_{L^2(\Omega_R)}, \quad t\in\Omega_R,
\end{equation}	
where $h$ (a ``filter function'') is a continuously differentiable
function on $\mathbb{R}^+$ satisfying

\[
h(t) = \left\{
  \begin{array}{l l}
    1, & \quad t\in\left[0,1/2\right],\\
    0, & \quad t\in\left[1,\infty\right),
  \end{array}
  \quad 0\le h(t)\le 1\; \mbox{for} \, t\in \mathbb{R^+}\right .
\]
Explicit  examples of suitable filter functions $h$ can be found
in \cite{S2012}. It is important for our later analysis that, as shown
in \cite{S2011}, $V_M x$ is (up to a constant) an optimal approximation in
the uniform norm, in the sense that,
\begin{equation}\label{eq:best}
\left\|x-V_Mx\right\|_{C(\Omega_R)}\leq c\inf_{p\in\mathbb{P}_{\left[M/2\right]}}\left\|x-p\right\|_{C(\Omega_R)},
\end{equation}
where $c$ is a generic constant, which may take different values at
different occurrences, and $\left[\cdot\right]$ denotes the floor
function.

In view of (\ref{eq:20}), for polynomially increasing $\phi, \beta$ and
for $x\in W^{\phi,\beta}$ we have

\[
\left\|x-V_Mx\right\|_{C(\Omega_R)}\leq c\left[M/2\right]^{-\nu}\leq cM^{-\nu}.
\]
On the other hand, for exponentially increasing $\beta$ and polynomially
increasing $\phi$ the theory \cite{R2000} suggests taking $h(t)\equiv
1, t\in\left[0,1\right]$. In this case the right-hand side of
\eqref{eq:best} has to be modified by multiplying by $\sqrt{M}$, and by
replacing $[M/2]$ by $M$, thus for $x\in W^{\phi,\beta}$ there holds

\[
\left\|x-V_Mx\right\|_{C(\Omega_R)}\leq c\sqrt{M}\inf_{p\in\mathbb{P}_M}\left\|x-p\right\|_{C(\Omega_R)}\leq c\sqrt{M}e^{-qM}.
	\]

	\section{Error estimation in uniform norm}

In this section we will estimate the uniform error of
approximation of $x$ by the polynomial $x^\epsilon_{\alpha,\lambda}$
given by (\ref{eq:19}).  It is clear that

\begin{small}
\begin{eqnarray}\label{eq:21}
\left\|x-x^\epsilon_{\alpha,\lambda}\right\|_{C(\Omega_R)}&=&\left\|x-R_{\alpha,M}T_{\lambda,M}y^\epsilon\right\|_{C(\Omega_R)}\\ \nonumber
&\leq& \left\|x-V_Mx\right\|_{C(\Omega_R)}
+\left\|V_Mx-R_{\alpha,M}U_My\right\|_{C(\Omega_R)}   \\ \nonumber
&+&\left\|(R_{\alpha,M}-R_{\alpha,M}T_{\lambda,M})U_My\right\|_{C(\Omega_R)}  \\ \nonumber
&+&\left\|R_{\alpha,M}T_{\lambda,M}\right\|_{C(\Omega_\rho)\rightarrow C(\Omega_R)}\left(\left\|U_My-y\right\|_{C(\Omega_\rho)}+\left\|y-y^\epsilon\right\|_{C(\Omega_\rho)}\right),
\end{eqnarray}
\end{small}
where $U_My$ is the same as $V_M x$ in \eqref{VMx} except that it is
an approximation on the larger sphere $\Omega_{\rho}$ instead of on
$\Omega_R$,

\begin{equation}\label{eq:UMy}
\left(U_My\right)(t)=\sum_{k=0}^M h\left(\frac{k}{M}\right)\sum_{j=1}^{2k+1}\frac{1}{\rho^2}Y_{k,j}\left(\frac{t}{\rho}\right)\left\langle Y_{k,j}\left(\frac{\cdot}{\rho}\right),y(\cdot)\right\rangle_{L^2(\Omega_\rho)}.
\end{equation}
It is natural to assume that
$\left\|U_My-y\right\|_{C(\Omega_\rho)}<\epsilon$, since otherwise data
noise is dominated by the approximation error and no regularization is
required. We also restrict ourselves to the case when
$\left\|x-V_Mx\right\|_{C(\Omega_R)}<c\epsilon$, otherwise the term
$\left\|x-V_Mx\right\|_{C(\Omega_R)}$, representing the error of
almost best polynomial approximation, will dominate the error bound.

Then the bound (\ref{eq:21}) can be reduced to the following one

\begin{eqnarray}\label{eq:22}
&&\left\|x-x^\epsilon_{\alpha,\lambda}\right\|_{C(\Omega_R)}\leq \left\|V_Mx-R_{\alpha,M}U_My\right\|_{C(\Omega_R)}   \\ \nonumber
&&+\left\|(R_{\alpha,M}-R_{\alpha,M}T_{\lambda,M})U_My\right\|_{C(\Omega_R)}+c(1+\left\|R_{\alpha,M}T_{\lambda,M}\right\|_{C(\Omega_\rho)\rightarrow C(\Omega_R)})\epsilon.
\end{eqnarray}


An estimate of the coefficient in the last term is given by the
following theorem.

\begin{theorem} Under the conditions of Theorem 1

\begin{scriptsize}
\[
\left\|R_{\alpha,M}T_{\lambda,M}\right\|_{C(\Omega_\rho)\rightarrow C(\Omega_R)} \leq  \frac{1}{R\rho}\max_{t\in\Omega_R}\left|\sum_{i=1}^{N}\omega_{i} \sum_{k=0}^M\frac{(2k+1)a_k}{4\pi(\alpha+a_k^2)(1+\lambda\beta_k^2)}   P_k\left(\frac{t \cdot t_{i}}{R\rho}\right)\right|,
\]
\end{scriptsize}
where $P_k$ are the Legendre polynomials of degree $k$.
\end{theorem}

\begin{proof}

In view of the definition (\ref{eq:8}) and (\ref{eq:Ra}), together
with \eqref{eq:orthog} with $R$ replaced by $\rho$ and the property
\eqref{eq:5} of the cubature rule, we can write
\begin{eqnarray*}
&& \left(R_{\alpha,M}T_{\lambda,M} y\right)(t)\\
&=&\sum_{k=0}^M \frac{a_k}{\alpha+a_k^2}\frac{1}{1+\lambda\beta_k^2}  \sum_{j=1}^{2k+1}\frac{1}{R}Y_{k,j}\left(\frac{t}{R}\right)\sum_{i=1}^N \omega_i\frac{1}{\rho}Y_{k,j}\left(\frac{t_i}{\rho}\right)y(t_i)\\
&=& \frac{1}{R\rho}
 \sum_{k=0}^{M}\frac{2k+1}{4\pi}  \frac{a_{k}}{\alpha+a_{k}^{2}}\frac{1}{1+\lambda\beta_{k}^{2}}
\sum_{i=1}^{N} \omega_{i}P_{k}\left(\frac{t \cdot t_{i}}{R\rho}\right)y(t_i),  \;\; \forall t\in \Omega_{R},
\end{eqnarray*}
where in the last step we used the addition theorem for the spherical
harmonics \cite{M1966}
\[
\sum_{j=1}^{2k+1}Y_{k,j}(x)Y_{k,j}(y)=\frac{2k+1}{4\pi}P_k(x\cdot y), \quad x,y\in\Omega_1.
\]
\qed
\end{proof}

Finally, we estimate the first two terms of the error bound (\ref{eq:22}).

\begin{theorem} Assume that the smoothness index functions $\phi, \psi$ are such that the functions $t/\phi(t)$,
and $t/\psi(t)$ are monotone. Then for $x\in W^{\phi,\beta}\bigcap
W^{\psi,a}$ and $Ax=y$ we have

\begin{eqnarray*}
\left\|V_Mx-R_{\alpha,M}U_My\right\|_{C(\Omega_R)}&+&\left\|(R_{\alpha,M}-R_{\alpha,M}T_{\lambda,M})U_My\right\|_{C(\Omega_R)} \\
&\leq&cM[\widehat{\psi}(\alpha)\left\|x\right\|_{W^{\psi,a}}+\widehat{\phi}(\lambda)\left\|x\right\|_{W^{\phi,\beta}}],
\end{eqnarray*}
where $\widehat{\phi}(\lambda)=\phi(\lambda),
\widehat{\psi}(\alpha)=\psi(\alpha)$, if $t/\phi(t)$ and $t/\psi(t)$ are
non-decreasing, and $\widehat{\phi}(\lambda)=\lambda,
\widehat{\psi}(\alpha)=\alpha$ otherwise.

\end{theorem}

\begin{proof}

It follows easily from \eqref{eq:2} and \eqref{eq:orthog} that
\[
a_k\left\langle\frac{1}{R}Y_{k,j}\left(\frac{\cdot}{R}\right),x\right\rangle_{L_2(\Omega_R)}=\left\langle\frac{1}{\rho}Y_{k,j}\left(\frac{\cdot}{\rho}\right),y\right\rangle_{L_2(\Omega_\rho)}.
\]
Hence from \eqref{VMx} and \eqref{eq:UMy} we obtain
\begin{eqnarray*}
\left(V_Mx\right)(t)&=&\sum_{k=0}^M h\left(\frac{k}{M}\right)\frac{1}{a_k}\sum_{j=1}^{2k+1}\frac{1}{R}Y_{k,j}\left(\frac{t}{R}\right)\left\langle \frac{1}{\rho}Y_{k,j}\left(\frac{\cdot}{\rho}\right),y\right\rangle_{L^2(\Omega_\rho)} \\
&=&\sum_{k=0}^M \frac{1}{a_k}\sum_{j=1}^{2k+1}\frac{1}{R}Y_{k,j}\left(\frac{t}{R}\right)\left\langle \frac{1}{\rho}Y_{k,j}\left(\frac{\cdot}{\rho}\right),\left(U_My\right)\right\rangle_{L^2(\Omega_\rho)}.
\end{eqnarray*}

Moreover, in view of \eqref{eq:22}, \eqref{eq:UMy} and (\ref{eq:5})
we also have the representations

\begin{small}
\begin{eqnarray*}
\left(R_{\alpha,M}U_My\right)(t)&=&\sum_{k=0}^M \frac{a_k}{\alpha+a_k^2}\sum_{j=1}^{2k+1}\frac{1}{R}Y_{k,j}\left(\frac{t}{R}\right)\sum_{i=1}^N \omega_i\frac{1}{\rho}Y_{k,j}\left(\frac{t_i}{\rho}\right)\left(U_My\right)(t_i) \\
&=&\sum_{k=0}^M \frac{a_k}{\alpha+a_k^2}\sum_{j=1}^{2k+1}\frac{1}{R}Y_{k,j}\left(\frac{t}{R}\right)\left\langle \frac{1}{\rho}Y_{k,j}\left(\frac{\cdot}{\rho}\right),\left(U_My\right)(\cdot)\right\rangle_{L^2(\Omega_\rho)},
\end{eqnarray*}
\end{small}
while from \eqref{eq:inner} and \eqref{eq:orthog} and two uses of
\eqref{eq:3} we obtain

\begin{scriptsize}
\begin{eqnarray*}
\left(R_{\alpha,M}T_{\lambda,M}U_My\right)(t)&=&\sum_{k=0}^M \frac{a_k}{\alpha+a_k^2}\frac{1}{1+\lambda\beta_k^2}\sum_{j=1}^{2k+1}\frac{1}{R}Y_{k,j}\left(\frac{t}{R}\right)\sum_{i=1}^N \omega_i\frac{1}{\rho}Y_{k,j}\left(\frac{t_i}{\rho}\right)\left(U_My\right)(t_i) \\
&=&\sum_{k=0}^M \frac{a_k}{\alpha+a_k^2}\frac{1}{1+\lambda\beta_k^2}\sum_{j=1}^{2k+1}\frac{1}{R}Y_{k,j}\left(\frac{t}{R}\right)\left\langle \frac{1}{\rho}Y_{k,j}\left(\frac{\cdot}{\rho}\right),\left(U_My\right)(\cdot)\right\rangle_{L^2(\Omega_\rho)}.
\end{eqnarray*}
\end{scriptsize}

Then we can write, using again \eqref{eq:UMy},

\begin{eqnarray*}
&&\left\|V_Mx-R_{\alpha,M}U_My\right\|_{C(\Omega_R)} \\
&=&\left\|\sum_{k=0}^M \left(\frac{1}{a_k}-\frac{a_k}{\alpha+a_k^2}\right)\sum_{j=1}^{2k+1}\frac{1}{R}Y_{k,j}\left(\frac{t}{R}\right)\left\langle \frac{1}{\rho}Y_{k,j}\left(\frac{\cdot}{\rho}\right),U_My\right\rangle_{L^2(\Omega_\rho)}\right\|_{C(\Omega_R)} \\
&=&\left\|\sum_{k=0}^M  h\left(\frac{k}{M}\right)\frac{\alpha}{a_k(\alpha+a_k^2)} \sum_{j=1}^{2k+1}\frac{1}{R}Y_{k,j}\left(\frac{t}{R}\right)\left\langle \frac{1}{\rho}Y_{k,j}\left(\frac{\cdot}{\rho}\right),y \right\rangle_{L^2(\Omega_\rho)}\right\|_{C(\Omega_R)}.
\end{eqnarray*}

Now using the property that $0 \leq h\left(\frac{k}{M}\right)\leq 1$ and
then the Nikolskii inequality (see, e.g., \cite{N2006},
Proposition 2.5), we obtain, for $x\in W^{\psi,a}$,

\begin{eqnarray*}
&&\left\|V_Mx-R_{\alpha,M}U_My\right\|_{C(\Omega_R)} \\
&\leq& \left\|\sum_{k=0}^M \frac{\alpha}{\alpha+a_k^2} \sum_{j=1}^{2k+1}\frac{1}{R}Y_{k,j}\left(\frac{t}{R}\right) \left\langle \frac{1}{R}Y_{k,j}\left(\frac{\cdot}{R}\right),x\right\rangle_{L^2(\Omega_R)}\right\|_{C(\Omega_R)} \\
&\leq& cM\left\|\sum_{k=0}^M \frac{\alpha}{\alpha+a_k^2} \sum_{j=1}^{2k+1}\frac{1}{R}Y_{k,j}\left(\frac{t}{R}\right) \left\langle \frac{1}{R}Y_{k,j}\left(\frac{\cdot}{R}\right),x\right\rangle_{L^2(\Omega_R)}\right\|_{L^2(\Omega_R)}\\
&= & cM\left\|\sum_{k=0}^M \frac{\alpha}{\alpha+a_k^2} \psi(a_{k}^{2}) \sum_{j=1}^{2k+1} \frac{\frac{1}{R}Y_{k,j}\left(\frac{t}{R}\right)\left\langle \frac{1}{R}Y_{k,j}\left(\frac{\cdot}{R}\right),x(\cdot)\right\rangle_{L^2(\Omega_R)}}{\psi(a_{k}^{2}) }\right\|_{L^2(\Omega_R)}\\
&\leq& cM\sup_{u\leq a_1^2} \left|\frac{\alpha}{\alpha+u}\psi(u)\right|\left\|x\right\|_{W^{\psi,a}}.
\end{eqnarray*}	

The proof can be completed by using Proposition 2.7 from \cite{L2013}. For convenience, we present the corresponding argument here. 

Let us consider two cases. First, we assume that $t/\psi(t)$ is a non-decreasing function. Then keeping in mind that the function $\psi$ is increasing, for $0<u<\alpha$ we obtain

$$\sup_{u<\alpha} \left|\frac{\alpha}{\alpha+u}\psi(u)\right|\leq \sup_{u<\alpha} \psi(u)< \psi(\alpha).$$ 

If $\alpha\leq u$ we can write 

\begin{eqnarray*}
\sup_{\alpha\leq u} \left|\frac{\alpha}{\alpha+u}\psi(u)\right|&\leq&\sup_{\alpha\leq u} \left|\frac{\alpha u}{\alpha+u}\frac{\psi(u)}{u}\right|\leq\sup_{\alpha\leq u} \left|\frac{\alpha u}{\alpha+u}\right| \sup_{\alpha\leq u} \left|\frac{\psi(u)}{u}\right| \\
&\leq& \frac{\alpha}{\inf_{\alpha\leq u}\frac{u}{\psi(u)}}\leq \frac{\alpha}{\frac{\alpha}{\psi(\alpha)}}=\psi(\alpha).
\end{eqnarray*}

Assume now that $t/\psi(t)$ is a decreasing function. Then

$$\sup_{u\leq a_1^2} \left|\frac{\alpha}{\alpha+u}\psi(u)\right|\leq \sup_{u\leq a_1^2} \left|\frac{\alpha u}{\alpha+u}\frac{\psi(u)}{u}\right|\leq \alpha\frac{\psi(a_1^2)}{a_1^2}\leq c\alpha.$$

Combining everything together we obtain

$$\left\|V_Mx-R_{\alpha,M}U_My\right\|_{C(\Omega_R)}\leq cM\widehat{\psi}(\alpha)\left\|x\right\|_{W^{\psi,a}}. $$ 

The second part of the regularization error can be bounded by
applying similar steps and using the assumption of the theorem that $x\in
W^{\phi,\beta}$. \qed
\end{proof}

From \eqref{eq:22} and Theorems 2 and 3 we can estimate the error
of the approximation of $x$ with the use of the proposed two-step
regularization scheme as follows:

\begin{eqnarray}\label{eq:23}
\left\|x-x^\epsilon_{\alpha,\lambda}\right\|_{C(\Omega_R)}&\leq& c
M[\widehat{\phi}(\lambda)\left\|x\right\|_{W^{\phi,\beta}}+\widehat{\psi}(\alpha)\left\|x\right\|_{W^{\psi,a}}]
\\ \nonumber
&+&c\max_{t\in\Omega_R}\sum_{k=0}^M\frac{(2k+1)a_k}{4\pi(a_k^2+\alpha)(1+\lambda\beta_k^2)}
\frac{1}{R\rho}\left|\sum_{i=1}^N \omega_i P_{k}\left(\frac{t\cdot
t_i}{R\rho}\right)\right| \epsilon.
\end{eqnarray}

From the error decomposition (\ref{eq:23}) one can easily obtain the error
bounds for the single parameter regularization schemes involved in the
two-step combination. If $\lambda=0$, which means that no noise
reduction has been done, we derive the following error bound in the
uniform norm:

\begin{eqnarray}\label{eq:24}
\left\|x-x^\epsilon_{\alpha,\lambda}\right\|_{C(\Omega_R)}&\leq& cM\widehat{\psi}(\alpha)\left\|x\right\|_{W^{\psi,a}} \\ \nonumber
&+&c\max_{t\in\Omega_R}\sum_{k=0}^M\frac{(2k+1)a_k}{(\alpha+a_k^2)}\frac{1}{R\rho}\left|\sum_{i=1}^N \omega_i P_{k}\left(\frac{t\cdot t_i}{R\rho}\right)\right| \epsilon.
\end{eqnarray}

 On the other hand, if $\alpha=0$, which means that direct inversion
after data pre-smoothing has been done, then

\begin{eqnarray}\label{eq:25}
\left\|x-x^\epsilon_{\alpha,\lambda}\right\|_{C(\Omega_R)}&\leq& cM\widehat{\phi}(\lambda)\left\|x\right\|_{W^{\phi,\beta}} \\ \nonumber
&+&c\max_{t\in\Omega_R}\sum_{k=0}^M\frac{2k+1}{4\pi a_k(1+\lambda\beta_k^2)}\frac{1}{R\rho}\left|\sum_{i=1}^N \omega_i P_{k}\left(\frac{t\cdot t_i}{R\rho}\right)\right|\epsilon.
\end{eqnarray}

Now we illustrate the potential advantage of the two-step
approximation compared to the single-parameter schemes. Assume that 
$\phi(t)=t^\mu$, and $\psi(t)=t^\nu$ for some $\mu,\nu \in (0,1)$. In
this case the error bound (\ref{eq:23}) is reduced to the following one

\begin{eqnarray}\label{eq:26}
\left\|x-x^\epsilon_{\alpha,\lambda}\right\|_{C(\Omega_R)}&\leq& cM(\lambda^\mu\left\|x\right\|_{W^{\phi,\beta}}+\alpha^\nu\left\|x\right\|_{W^{\psi,a}}) \\ \nonumber
&+&c \max_{t\in\Omega_R}\sum_{k=0}^M\frac{(2k+1)a_k}{(\alpha+a_k^2)(1+\lambda\beta_k^2)}\frac{1}{R\rho}\left|\sum_{i=1}^N \omega_i P_{k}\left(\frac{t\cdot t_i}{R\rho}\right)\right|\epsilon.
\end{eqnarray}
Consider for simplicity the situation in which for fixed $\mu$ and $\nu$,
$\alpha=\lambda^{\mu/\nu}$.  Then the first terms of \eqref{eq:24},
\eqref{eq:25} and \eqref{eq:26} are essentially the same, whereas the last
term in \eqref{eq:26} is obviously smaller than that in either
\eqref{eq:24} or \eqref{eq:25}.



\section{Numerical illustrations}

In all our experiments the regularization parameters $\lambda, \alpha$ are
chosen by means of the so-called quasi-optimality criterion. This
heuristic approach was originally proposed in \cite{T1965} and has been
recently advocated in \cite{Ki2008}.

In the case of the noise reduction method (\ref{eq:7}), the quasi-optimality criterion chooses a regularization parameter $\lambda=\lambda_l$ from a set

\begin{equation}\label{eq:27}
Q_L^\lambda=\left\{\lambda=\lambda_i=\lambda_0q^i,\ i=0, 1, 2,\ldots,L\right\},\ \ q>1,
\end{equation}
such that

\begin{scriptsize}
$$\left\|A_M^{-1}(T_{\lambda_l,M}y^\epsilon-T_{\lambda_{l-1},M}y^\epsilon)\right\|=\min\left\{\left\|A_M^{-1}(T_{\lambda_i,M}y^\epsilon-T_{\lambda_{i-1},M}y^\epsilon)\right\|,\ i=0, 1,\ldots,L\right\},$$
\end{scriptsize}
where $\left\|\cdot\right\|$ denotes the uniform norm in the space $C=C(\Omega_R)$ of continuous functions. 

In the similar way we can apply the quasi-optimality criterion in the regularized collocation method. Then a regularization parameter $\alpha=\alpha_k$ is chosen from a set

\begin{equation}\label{eq:28}
Q_K^\alpha=\left\{\alpha=\alpha_j=\alpha_0r^j,\ j=0, 1, 2,\ldots,K\right\},\ \ r>1,
\end{equation}
in such a way that

$$\left\|R_{\alpha_k, M}y^\epsilon-R_{\alpha_{k-1}, M}y^\epsilon\right\|=\min\left\{\left\|R_{\alpha_j, M}y^\epsilon-R_{\alpha_{j-1}, M}y^\epsilon\right\|,\ j=0, 1,\ldots,K\right\}.$$

For the two-step regularization the quasi-optimality criterion can be implemented as follows. First, for every $\alpha=\alpha_j \in Q_K^\alpha$ we choose $\lambda=\lambda_l=\lambda(\alpha_j)$ from the set (\ref{eq:27}) such that

$$\left\|x^\epsilon_{\alpha_j,\lambda_l}-x^\epsilon_{\alpha_j,\lambda_{l-1}}\right\|=\min\left\{\left\|x^\epsilon_{\alpha_j,\lambda_i}-x^\epsilon_{\alpha_j,\lambda_{i-1}}\right\|,\ i=0, 1,\ldots,L\right\}.$$
Next, we apply the quasi-optimality criterion to the sequence $\left\{x^\epsilon_{\lambda(\alpha_j),\alpha_j}\right\}$ parametrized by $\alpha_j \in Q_K^\alpha$. More specifically, we select $\alpha_k \in Q_K^\alpha$ such that

\begin{scriptsize}
$$\left\|x^\epsilon_{\alpha_k, \lambda(\alpha_k)}-x^\epsilon_{\alpha_{k-1}, \lambda(\alpha_{k-1})}\right\|=\min\left\{\left\|x^\epsilon_{\alpha_j, \lambda(\alpha_j)}-x^\epsilon_{\alpha_{j-1}, \lambda(\alpha_{j-1})}\right\|,\ j=0, 1,\ldots,K\right\}.$$
\end{scriptsize}
Then, a regularized approximate solution $x^\epsilon_{\alpha,\lambda}$ of our choice is defined by (\ref{eq:19}) with $\lambda=\lambda(\alpha_k)$ and $\alpha=\alpha_k$.

In all our experiments we follow \cite{G2009, N2013} and assume that $\left\{t_i\right\}_{i=1}^N$ is the set of Gauss-Legendre points, for which the positive quadrature weights are known analytically. In this case $N=2(M+1)^2$, and we take $M=30$.
	
The data are simulated in the following way. First we generate a spherical function

\[
y(t)=\sum_{k=0}^M a_k\sum_{j=1}^{2k+1}\widehat{x}_{k,j}\frac{1}{\rho}Y_{k,j}\left(\frac{t}{\rho}\right),\ \ t\in \Omega_\rho,
\]
where $\widehat{x}_{k,j}=(k+1/2)^{-\upsilon} g_{k,j}, \ k=0,...,M, \
j=1,...2k+1$, $\upsilon>0$, and $g_{k,j}$ are uniformly distributed random
values from $\left[-1,1\right]$. This means that the solutions of the simulated problems have Sobolev smoothness $\upsilon$, and  $y$ can be seen as a
result of the action of the operator $A$ on a function $x$ from the
spherical Sobolev space $H_2^\upsilon$. A noisy spherical function
$y^\epsilon$ is simulated by adding a Gaussian white noise of intensity
$\epsilon=0.05$ to the values of the initial function $y$ at the points
$\left\{t_i\right\}_{i=1}^N$.

\begin{figure}
        \centering
        \begin{subfigure}[b]{0.42\linewidth}
               \includegraphics[width=\textwidth]{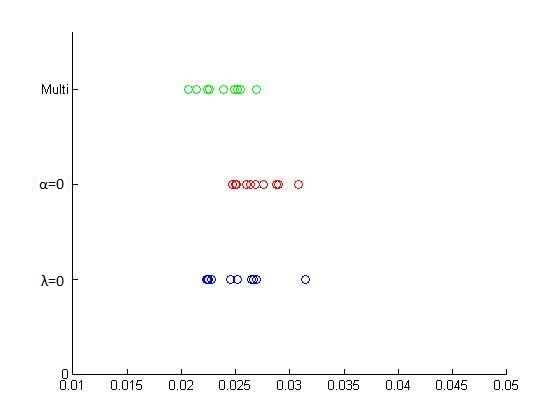}
                \caption{\small{Relative errors for the case $a_k=(1.48)^{-k}, \upsilon=3/2, \beta_k^2=a_k^{-1}, k=1,2,\ldots,M$.}}
        \end{subfigure}%
      \quad
        \begin{subfigure}[b]{0.42\linewidth}
               \includegraphics[width=\textwidth]{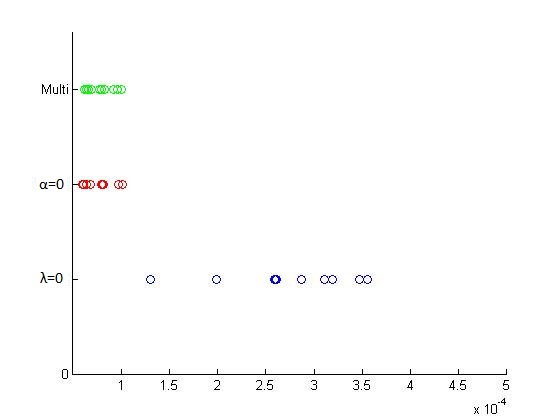}
                \caption{\small{Relative errors for the case $a_k=(1.48)^{-k}, \upsilon=11/2, \beta_k^2=a_k^{-1}(k+1/2)^{7/2}, k=1,2,\ldots,M$.}}
        \end{subfigure}
    \\
		        \begin{subfigure}[b]{0.42\linewidth}
                \includegraphics[width=\textwidth]{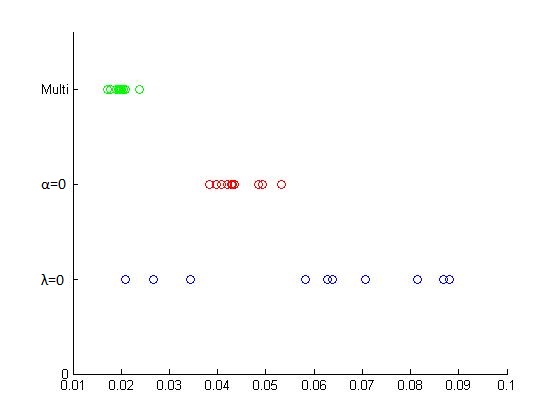}
                \caption{\small{Relative errors for the case $a_k=(k+1)^{-2}, \upsilon=3/2, \beta_k^2=a_k^{-1}, k=1,2,\ldots,M$.}}
        \end{subfigure}%
      \quad
        \begin{subfigure}[b]{0.42\linewidth}
                \includegraphics[width=\textwidth]{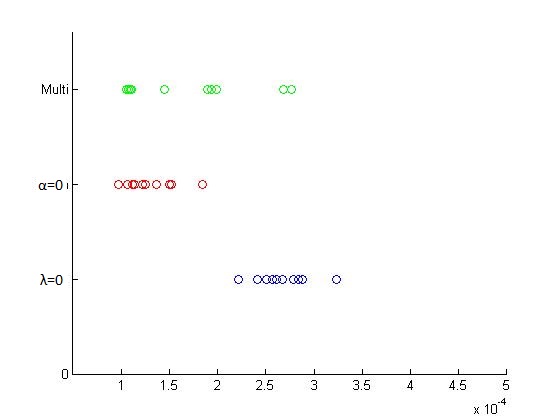}
              \caption{\small{Relative errors for the case $a_k=(k+1)^{-2}, \upsilon=11/2, \beta_k^2=a_k^{-1}(k+1/2)^{7/2}, k=1,2,\ldots,M$.}}
        \end{subfigure}
		\\
        \begin{subfigure}[b]{0.42\linewidth}
               \includegraphics[width=\textwidth]{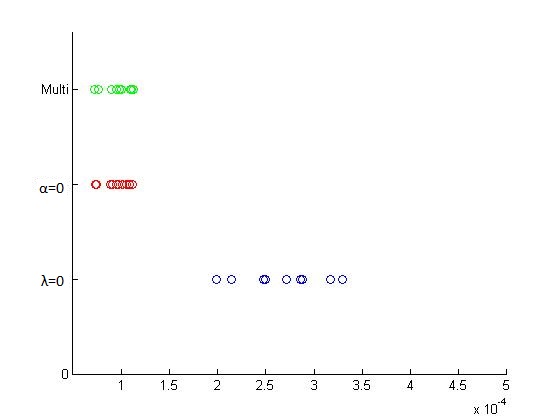}
                \caption{\small{Relative errors for the case $a_k=(k+1)^{-2}, \upsilon=11/2, \beta_k^2=a_k^{-1}(k+1/2)^{11/2}, k=1,2,\ldots,M$.}}
        \end{subfigure}
        \caption{Numerical illustrations. The comparison of the performance of single-parameter regularizations and their combination (Multi). In each of the figures the results corresponding to the two-step regularization are displayed at the top, the case of noise reduction with direct inversion is shown in the middle, and the case of regularized collocation corresponds to the bottom row. The horizontal axes show values of the relative errors (the vertical axes have no significance)}
				\label{fig:1}
\end{figure}

To assess the performance of the considered schemes we measure the relative error

\[
\frac{\left\|x-x^\epsilon_{\alpha,\lambda}\right\|}{\left\|x\right\|},
	\]
where

$$x=\sum_{k=0}^M\sum_{j=1}^{2k+1}\widehat{x}_{k,j}\frac{1}{R}Y_{k,j}\left(\frac{\cdot}{R}\right),$$
and the approximation $x^\epsilon_{\alpha,\lambda}$ is given by (\ref{eq:19}).

The results are displayed in Figure \ref{fig:1}, where each subfigure corresponds to different values of $a_k, \upsilon$, and penalization weights $\beta_k$. In Figure \ref{fig:1} each circle exhibits a value of the relative error in solving the problem with one of 10 simulated data, for each of three methods: regularized collocation method corresponding to the case when $\lambda=0$, noise reduction method with direct inversion (the case when $\alpha=0$), and the two-step regularization method. Note that such a form of graphical representation of the performance of different regularization algorithms is rather common (see, e.g., \cite{H1998}).

In our experiments we tried to cover different degrees of ill-posedness of (\ref{eq:1}), (\ref{eq:2}), i.e. the so-called moderate ill-posedness when the spherical symbol tends to zero polynomially (Figures \ref{fig:1}c, \ref{fig:1}d, \ref{fig:1}e), and severely ill-posed case with exponential symbol decrease (Figures \ref{fig:1}a, \ref{fig:1}b). Note that the spherical pseudo-differential operators with polynomial symbol have been studied in \cite{L2006}, while the operators with exponentialy decreasing symbol, appearing in geoscience, have been discussed in the Introduction.

It is instructive to look at Figures \ref{fig:1}a, \ref{fig:1}b and \ref{fig:1}c, \ref{fig:1}d from the view point of \cite{P2014}, where \textit{a priori} choice of the penalization weights $\beta_k$ has been discussed. For the situations described in the above mentioned figures the choice of $\beta_k$  suggested in  \cite{P2014} would be $\beta_k^2=a_k^{-1}(k+1)^\upsilon$. But such a choice requires knowledge of the smoothness $\upsilon$ of the unknown solutions. In practice the smooothness $\upsilon$ is usually unknown, and the  Figures \ref{fig:1}a--\ref{fig:1}d correspond to the case of ``underpenalization''. 

As can be seen from Figures \ref{fig:1}a--\ref{fig:1}d in such a case it is not clear \textit{a priori}  which of the considered single-parameter regularizations performs better. At the same time, these figures show that the proposed two-step method automatically follows the leader. The situation does not change when the penalization weights are chosen according to \cite{P2014} (see Figure \ref{fig:1}e). 

In all our experiments the choice of the regularization parameters was done by using the quasi-optimality criterion described above with $\alpha_0=\lambda_0=1.78e-5$, $q=r=1.25$, and $L=K=50$. Note also that we add the term $\lambda=\alpha=0$ to each of the sets (\ref{eq:27}), and (\ref{eq:28}).

These numerical results confirm the theoretical conclusion that the constructed two-step regularization methods can perform similarly to (or even better than) the best of single-parameter schemes.

\section{Conclusion}

We studied a two-parameter regularization scheme, which is a combination of the inversion after a data presmoothing and regularized collocation. In principle, each of the combined methods can be used alone as a regularization method. However, their performance very much depends on the interplay between the degree of ill-posedness of (\ref{eq:1}), (\ref{eq:2}), coded in the rate of increase of the spherical symbol $\left\{a_k\right\}$, and the smoothness of the solution.

For example, from Figure \ref{fig:1} one may conclude that for a severely ill-posed problem and moderately smooth solution (Figure \ref{fig:1}a) the regularized collocation performs a bit better than the inversion after a data presmoothing, but for a moderately ill-posed problem (Figure \ref{fig:1}c) this is not the case.

Since for ill-posed problems the smoothness of their solutions is usually unknown, it is not clear which of the single parameter regularizations is the better one. In this situation the combination of the single parameter regularizations, resulting in a two-parameter scheme, can be seen as a method of choice. The above figures illustrate this point.

At the end of Section 4 we have used Theorems 2 and 3 to demonstrate theoretically that the two-parameter scheme (\ref{eq:19}) is able to provide a better error bound than ones for the schemes involved in the combination. However, this demonstration is limited to an \textit{a priori} choice of the regularization parameters, requiring knowledge of the solution smoothness, and therefore seldom usable in practice.

On the other hand, in Section 5 the above mentioned ability of the two-parameter scheme has been demonstrated for an \textit{a posteriori} choice of the regularization parameters, where no knowledge of the solution smoothness has been used. At the same time, \textit{a posteriori} parameter choice employed in Section 5 is of a heuristic nature. We hope that this paper will stimulate research towards theoretical justification of \textit{a posteriori} rules for multi-parameter regularization.

\begin{acknowledgements}
The research of the first author was partially supported by the SRFDP fund
20120171120005 of China, Guangdong Provincial Government of China through
the ``Computational Science Innovative Research Team" program and
Guangdong Province Key Laboratory of Computational Science at the Sun
Yat-sen University. The second and the fourth authors are supported by the
Austrian Fonds Zur Forderung der Wissenscheftlichen Forschung (FWF), grant
P25424. The third author acknowledges the support of the Australian
Research Council, project DP 120101816.
\end{acknowledgements}



\end{document}